\documentclass[11pt]{amsart}
\usepackage{amssymb,color}
\usepackage{tikz}
\theoremstyle{plain}
\newtheorem{theorem}{Theorem}[section]

\newtheorem{remark}{Remark}[section]

\newtheorem{proposition}{Proposition}[section]

\theoremstyle{definition}
 
\theoremstyle{remark}

\allowdisplaybreaks

\long\def\symbolfootnote[#1]#2{\begingroup
	\def\thefootnote{\fnsymbol{footnote}}\footnote[#1]{#2}\endgroup}
\begin{document}
	\title[Parametric anisotropic $(p,2)$-equations]
	{Constant sign and nodal solutions for \\ parametric anisotropic $(p,2)$-equations}
	\author[N.S. Papageorgiou, D.D. Repov\v{s} and C. Vetro]{N. S. Papageorgiou, D. D. Repov\v{s} and C. Vetro}
	\address[N.S. Papageorgiou]{Department of Mathematics, National Technical University, Zografou campus, 15780, Athens, Greece}
	\email{npapg@math.ntua.gr}
	\address[D. D. Repov\v{s}]{Faculty of Education and Faculty of Mathematics and Physics,
		University of Ljubljana \& Institute of Mathematics, Physics and Mechanics,
		SI-1000, Ljubljana, Slovenia}
	\email{dusan.repovs@guest.arnes.si}
	\address[C. Vetro]{Department of Mathematics and Computer Science, University of Palermo, Via Archirafi 34, 90123, Palermo, Italy}
	\email{calogero.vetro@unipa.it}	
	\maketitle

\begin{abstract}
We consider an anisotropic $(p,2)$-equation, with a parametric and superlinear reaction term. We show that for all small values of the parameter the problem has at least five nontrivial smooth solutions, four with constant sign and the fifth nodal (sign-changing). 
The proofs use tools from critical point theory, truncation and comparison techniques, and critical groups.
\end{abstract}

\section{Introduction}\label{sec:1}

Let $\Omega \subseteq \mathbb{R}^N$ be a bounded domain with a $C^2$-boundary $\partial \Omega$. In this paper we study the following anisotropic $(p,2)$-equation
\begin{equation}\label{eq0}\tag{$P_\lambda$} 
\begin{cases}-\Delta_{p(z)} u(z)-\Delta u(z)=\lambda f(z,u(z))  & \mbox{in } \Omega,\\ u \Big|_{\partial \Omega} =0, \, \lambda>0.& \end{cases}
\end{equation}
In this problem the exponent $p: \overline{\Omega} \to (1,+\infty)$ is Lipschitz continuous and $2< p_-=\min\limits_{\overline{\Omega}}p$. By    $\Delta_{p(z)}$ we denote the variable exponent (anisotropic) $p$-Laplacian, defined by
$$\Delta_{p(z)} u= \mbox{div }(|\nabla u|^{p(z)-2}\nabla u) \quad \mbox{ for all } u \in W_0^{1,p(z)}(\Omega).$$

The reaction of the problem is parametric, with $\lambda>0$ being the parameter. The function $f(z,x)$ is  measurable in $z \in \Omega$, continuous in $x \in \mathbb{R}$. We assume that $f(z,\cdot)$ is $(p_+-1)$-superlinear as $x \to \pm \infty$ ($p_+=\max\limits_{\overline{\Omega}}p$) but without satisfying the usual in such cases Ambrosetti-Rabinowitz condition (the AR-condition for short). Our goal is to prove a multiplicity theorem for problem \eqref{eq0} providing sign information for all the solutions produced. Using variational tools from the critical point theory, together with suitable truncation and comparison techniques and also Morse Theory (critical groups), we show that for all small values of the parameter $\lambda>0$ the problem has at least five nontrivial smooth solutions (four of constant sign and the fifth nodal (sign-changing)).

\begin{theorem}\label{theo11}
If hypotheses $H_0$, $H_1$ hold, then there exists $\lambda^\ast>0$ such that for all $\lambda \in (0,\lambda^\ast)$ problem \eqref{eq0} has at least five nontrivial smooth solutions $u_0, \widehat{u}\in {\rm int \,}C_+$, $v_0, \widehat{v}\in -{\rm int \,}C_+$, $y_0 \in  C_0^1(\overline{\Omega})$ nodal. 
\end{theorem}

\begin{remark}
	 The hypotheses $H_0, H_1$ and spaces $C_+, C_0^1(\overline{\Omega})$  are defined in the next section.
	 We stress that the above multiplicity theorem provides sign information for all the solutions.
\end{remark}

Anisotropic equations arise in a variety of models of physical processes. We mention the works of Bahrouni-R\u{a}dulescu-Repov\v{s} \cite{Ref1} (transonic flow problems), R\.{u}\v{z}i\v{c}ka \cite{Ref19} (electrorheological and magnetorheological fluids),  Zhikov \cite{Ref24} (nonlinear elasticity theory), and Agarwal-Alghamdi-Gala-Ragusa \cite{Ref0}, Ragusa-Tachikawa \cite{RefRT} (double phase problems).  Recently there have been some existence and multiplicity results for various types of $(p,q)$-equations with nonstandard growth. We refer to the works of Gasi\'nski-Papageorgiou \cite{Ref7}, R\u{a}dulescu-Repov\v{s} \cite{Ref18}, R\u{a}dulescu  \cite{Ref17}, Papageorgiou-R\u{a}dulescu-Repov\v{s} \cite{Ref12}, Papageorgiou-Scapellato \cite{Ref14}, Papageorgiou-Vetro \cite{Ref15}, Zhang-R\u{a}dulescu \cite{Ref23}. They produce at most three nontrivial solutions, but no nodal solutions. We also mention the recent isotropic works of Li-Rong-Liang \cite{Ref10}, Papageorgiou-Vetro-Vetro \cite{Ref16} producing two positive solutions for $(p,2)$- and $(p,q)$-equations respectively,
and the recent work of Papageorgiou-Scapellato \cite{Ref00} who consider a different class of parametric equations (superlinear perturbations of the standard eigenvalue problem) and produce seven solutions, all with sign information.

\section{Mathematical Background - Hypotheses}\label{sec:2}

The analysis of problem \eqref{eq0} requires the use of Lebesgue and Sobolev spaces with variable exponents. A comprehensive treatment of such spaces can be found in the book of Diening-Hajulehto-H\"ast\"o-R\.{u}\v{z}i\v{c}ka \cite{Ref2}. 

Given $q \in C(\overline{\Omega})$, we define
$$q_-=\min\limits_{\overline{\Omega}}q \mbox{ and } q_+=\max\limits_{\overline{\Omega}}q .$$

Let $E_1=\{q \in C(\overline{\Omega}) \, : \, 1<q_-  \}$ and $M(\Omega)=\{u : \Omega \to \mathbb{R} \mbox{ measurable}\}$ (as usual we identify two measurable functions which differ only on a Lebesgue null set). Given $q \in E_1$, we define the variable exponent Lebesgue space $L^{q(z)}(\Omega)$ as follows
$$ L^{q(z)}(\Omega)=\left\{ u \in M(\Omega)\, : \, \int_\Omega |u(z)|^{q(z)}dz < \infty   \right\}.  $$

This vector space is equipped with the so-called ``Luxemburg norm'' $\|\cdot \|_{q(z)}$ defined by
$$\|u \|_{q(z)}= \inf \left[\lambda >0 \, : \,   \int_\Omega\left|\frac{u(z)}{\lambda}\right|^{q(z)}dz\leq 1  \right].$$

Then $L^{q(z)}(\Omega)$ becomes a separable, uniformly convex (hence also reflexive) Banach space.  The reflexivity of these spaces leads to the reflexivity of the corresponding Sobolev spaces, which we introduce below. In reflexive Banach spaces bounded sequences have $w$-convergent subsequences (Eberlein-\v{S}mulian theorem). We will be using this fact repeatedly.  The dual $L^{q(z)}(\Omega)^\ast$ is given by $L^{q^\prime(z)}(\Omega)$ with $q^\prime \in E_1$ defined by $q^\prime(z)=\frac{q(z)}{q(z)-1}$ for all $z \in \overline{\Omega}$ (that is, $\frac{1}{q(z)}+\frac{1}{q^\prime (z)}=1$ for all $z \in \overline{\Omega}$). Also we have the following version of H\"older's inequality
$$\int_\Omega |u(z)h(z)|dz \leq \left[\frac{1}{q_-}+\frac{1}{q^\prime_-}\right]\|u\|_{q(z)}\|h\|_{q^\prime(z)} \quad \mbox{for all $u \in L^{q(z)}(\Omega)$, all $h \in L^{q^\prime(z)}(\Omega)$. }$$

If $q_1,q_2 \in E_1$ and $q_1(z) \leq q_2(z)$ for all $z \in \overline{\Omega}$, then $L^{q_2(z)}(\Omega) \hookrightarrow L^{q_1(z)}(\Omega)$ continuously.

Now that we have variable exponent Lebesgue spaces, we can define variable exponent Sobolev spaces. So, if $q \in E_1$, then we define
$$ W^{1,q(z)}(\Omega)=\left\{ u \in L^{q(z)}(\Omega)\, : \, |\nabla u| \in L^{q(z)}(\Omega)    \right\},  $$
with $\nabla u$ being the weak gradient of $u$. This Sobolev space is equipped with the norm 
$$\|u\|_{1,q(z)}=\|u\|_{q(z)}+ \|\nabla u\|_{q(z)}\quad \mbox{for all $u \in W^{1,q(z)}(\Omega)$.}$$

When $q \in E_1$ is Lipschitz continuous (that is, $q \in E_1 \cap C^{0,1}(\overline{\Omega})$), then we define the Dirichlet anisotropic Sobolev space $ W_0^{1,q(z)}(\Omega)$ by
$$ W_0^{1,q(z)}(\Omega)=\overline{C_c^\infty(\Omega)}^{\|\cdot\|_{1,q(z)}}.$$

Both spaces $ W^{1,q(z)}(\Omega)$ and $ W_0^{1,q(z)}(\Omega)$ are separable and uniformly convex (hence reflexive) Banach spaces.

If $q \in E_1$, then we define the critical Sobolev exponent corresponding to $q(\cdot)$ by setting 
$$q^\ast(z)=\begin{cases} \dfrac{Nq(z)}{N-q(z)} & \mbox{if } q(z) <N,\\ +\infty & \mbox{if }N \leq q(z).\end{cases}$$

Suppose that $q,r \in C(\overline{\Omega})$, $1<q_-,r_+<N$ and $1 \leq r(z) \leq q^\ast(z)$ for all $z \in \overline{\Omega}$ (resp. $1 \leq r(z) < q^\ast(z)$ for all $z \in \overline{\Omega}$). Then the anisotropic Sobolev embedding theorem says that 
$$ W^{1,q(z)}(\Omega) \hookrightarrow L^{r(z)}(\Omega) \mbox{ continuously}$$
$$ (\mbox{resp. }W^{1,q(z)}(\Omega) \hookrightarrow L^{r(z)}(\Omega) \mbox{ compactly}).$$

The same embedding theorem remains true also for $W_0^{1,q(z)}(\Omega)$ provided $q \in E_1 \cap C^{0,1}(\overline{\Omega})$. Moreover, in this case the Poincar\'e inequality is true, namely we can find $\widehat{c}>0$ such that 
$$ \|u\|_{q(z)}\leq \widehat{c}\, \|\nabla u\|_{q(z)} \quad \mbox{for all $u \in W_0^{1,q(z)}(\Omega)$.}$$

This means that on the anisotropic Sobolev space $W_0^{1,q(z)}(\Omega)$ we can consider the equivalent norm 
$$\|u\|_{1,q(z)}=  \|\nabla u\|_{q(z)}\quad \mbox{for all $u \in W_0^{1,q(z)}(\Omega)$.}$$

The following modular function is very helpful in the study of the anisotropic Lebesgue and Sobolev spaces. So, let $q \in E_1$. We define
$$\rho_q(u)=\int_\Omega |u(z)|^{q(z)}dz \quad \mbox{for all $u \in L^{q(z)}(\Omega)$.}$$

For $u \in W^{1,q(z)}(\Omega)$, we define $\rho_q(\nabla u)=\rho_q(|\nabla u|)$.

The modular function $\rho_q(\cdot)$ and the Luxemburg $\|\cdot \|_{q(z)}$ are closely related.

\begin{proposition}
	\label{prop1} If $q \in E_1$ and $\{u_n,u \}_{n \in \mathbb{N}} \subseteq L^{q(z)}(\Omega)$, then 
\begin{itemize}
	\item[(a)] for all $\lambda>0$ we have
	$$\|u\|_{q(z)}=\lambda \mbox{ if and only if } \rho_q\left(\frac{u}{\lambda}\right)=1;$$ 
		\item[(b)] $\|u\|_{q(z)}<1 \Leftrightarrow  \|u\|_{q(z)}^{q_+}\leq \rho_q(u) \leq  \|u\|_{q(z)}^{q_-},$\\ \\$\|u\|_{q(z)}>1 \Leftrightarrow  \|u\|_{q(z)}^{q_-}\leq \rho_q(u) \leq  \|u\|_{q(z)}^{q_+};$\\
			\item[(c)]  $\|u_n\|_{q(z)}\to 0  \Leftrightarrow \rho_q(u_n)\to 0;$\\
			\item[(d)]  $\|u_n\|_{q(z)}\to \infty \Leftrightarrow \rho_q(u_n)\to \infty.$
\end{itemize}	
\end{proposition}

Suppose that $q \in E_1 \cap C^{0,1}(\overline{\Omega})$. We have
$$W_0^{1,q(z)}(\Omega)^\ast=W^{-1,q^\prime(z)}(\Omega).$$

Then we introduce the operator $A_{q(z)}: W_0^{1,q(z)}(\Omega) \to  W^{-1,q^\prime(z)}(\Omega)$  defined by
$$\langle A_{q(z)}(u),h \rangle = \int_\Omega |\nabla u(z)|^{q(z)-2}(\nabla u, \nabla h)_{\mathbb{R}^N}dz \quad \mbox{for all } u,h \in W_0^{1,q(z)}(\Omega).$$
The next proposition summarizes the main properties of this operator (see Gasi\'nski-Papageorgiou \cite{Ref8}, Proposition 2.5, and R\u{a}dulescu-Repov\v{s} \cite{Ref18}, p. 40).

\begin{proposition}\label{prop2}
If $q \in E_1 \cap C^{0,1}(\overline{\Omega})$ and $A_{q(z)}: W_0^{1,q(z)}(\Omega) \to W^{-1,q^\prime(z)}(\Omega)$ is defined as above, then $A_{q(z)}(\cdot)$ is bounded $($maps bounded sets to bounded sets$)$, continuous, strictly monotone $($hence also maximal monotone$)$ and of type $(S)_+$  $($that is, if $u_n  \xrightarrow{w} u $ in $W_0^{1,q(z)}(\Omega)$ and $\limsup\limits_{n \to \infty} \langle A_{q(z)}(u_n), u_n - u \rangle \leq 0$, then $u_n \to u$ in $W_0^{1,q(z)}(\Omega)).$ 
\end{proposition}

Given $x \in \mathbb{R}$, we set $x^\pm =\max \{\pm x,0\}$. Then for $u \in W_0^{1,q(z)}(\Omega)$, we define  $u^\pm(z)=u(z)^\pm$ for all $z \in \Omega$. We know that 
$$u^\pm \in W_0^{1,q(z)}(\Omega), \, u=u^+-u^-, \, |u|=u^++u^-.$$

If $u,v: \Omega \to \mathbb{R}$ are measurable functions such that $u(z) \leq v(z)$ for a.a. $z \in \Omega$, then we define $[u,v]=\{y \in W_0^{1,q(z)}(\Omega)\, : \, u(z)\leq y(z)\leq v(z) \mbox{ for a.a. }z \in \Omega \}$ and $[u)=\{y \in W_0^{1,q(z)}(\Omega)\, : \, u(z)\leq y(z) \mbox{ for a.a. }z \in \Omega \}$.

We write $u \preceq v$ if and only if for every compact $K \subseteq \Omega$, we have $0 <c_K \leq v(z)-u(z)$ for a.a. $z \in K$. Evidently, if $u,v \in C(\Omega)$ and $u(z) <v(z)$ for all $z \in \Omega$, then $u \preceq v$.

Besides the anisotropic Lebesgue and Sobolev spaces, we will also use the ordered Banach space $C_0^1(\overline{\Omega})=\{u \in C^1(\overline{\Omega}) \, : \, u \big|_{\partial \Omega}=0\}$. The positive (order) cone of $C_0^1(\overline{\Omega})$ is
$C_+=\left\{u \in C_0^1(\overline{\Omega}) \, : \, u(z) \geq 0 \mbox{ for all } z \in \overline{\Omega}\right\}$. This cone has a nonempty interior given by
$$\mbox{int }C_+=\left\{u \in C_+ \, : \, u(z) > 0 \mbox{ for all } z \in  \Omega , \quad \frac{\partial u}{\partial n} \Big|_{\partial \Omega}<0\right\},$$
with $n(\cdot)$ being the outward unit normal on $\partial \Omega$. 

Suppose $X$ is a Banach space and $\varphi \in C^1(X)$. We set
$$ K_\varphi  = \{u \in X : \varphi'(u) =0 \} \quad \mbox{(the critical set of $\varphi$).}$$

We say that $\varphi(\cdot)$ satisfies the ``$C$-condition'', if it has the following property: 
\\

``Every sequence $\{u_n\}_{n \in \mathbb{N}} \subseteq X$ such that \begin{align*}&\{\varphi(u_n)\}_{n \in \mathbb{N}} \subseteq \mathbb{R} \mbox{ is bounded, and}\\ & (1 + \|u_n\|_X) \varphi'(u_n) \to 0 \mbox{ in }X^\ast \mbox{  as } n \to \infty,\end{align*}
\hskip .6cm  admits a strongly convergent subsequence''. 
\\

Given $c \in \mathbb{R}$, we set $\varphi^c  = \{u \in X : \varphi(u) \leq c \}$.

Suppose $Y_2 \subseteq Y_1 \subseteq X$. For every $k \in \mathbb{N}_0=\mathbb{N}\cup \{0\}$, by $H_k(Y_1,Y_2)$ we denote the $k^{th}$-singular homology group with integer coefficients for the pair $(Y_1,Y_2)$. Let $u \in K_\varphi$ be isolated and $c=\varphi(u)$. Then the critical groups of $\varphi$ at $u$ are defined by
$$C_k(\varphi,u)=H_k(\varphi^c \cap U, \varphi^c \cap U \setminus \{u\}) \quad \mbox{for all } k \in \mathbb{N}_0,$$
where $U$ is an open neighborhood of $u$ such that $K_\varphi \cap \varphi^c \cap U  = \{u\}$. The excision property of singular homology implies that this definition is independent of the choice of the  isolating neighborhood $U$.

In the sequel, for economy in the notation, by $\|\cdot\|$ we will denote the norm of the Sobolev space $W_0^{1,p(z)}(\Omega)$ ($p \in E_1 \cap C^{0,1}(\overline{\Omega})$). On account of the Poincar\'e inequality mentioned earlier, we have 
$$\|u\| =  \|\nabla u\|_{p(z)}\quad \mbox{for all $u \in W_0^{1,p(z)}(\Omega)$.}$$

Now we are ready to introduce our hypotheses on the data of problem \eqref{eq0}.

\medskip
\noindent $H_0$: $p \in C^{0,1}(\overline{\Omega})$ and $2<p(z) <N$ for all $z \in \overline{\Omega}$. 
\\
\noindent $H_1$: $f:\Omega \times \mathbb{R} \to \mathbb{R}$ is a Carath\'eodory function such that $f(z,0)=0$ for a.a. $z \in \Omega$, and 
\begin{itemize}
	\item[$(i)$] $|f (z,x)| \leq a(z) [1+|x|^{r(z)-1}]$ for a.a. $z \in \Omega$, all $x \in \mathbb{R}$, with $a \in L^\infty(\Omega)$, $ r \in C(\overline{\Omega})$ with $p(z)< r(z) < p_-^\ast$ for all $z \in \overline{\Omega}$;
	\item[$(ii)$] if $F(z,x)=\int_0^x f(z,s)ds$, then
	$$\lim_{x \to \pm \infty} \frac{F(z,x)}{x^{p_+}}=+\infty \mbox{ uniformly for a.a. $z \in \Omega$;}$$
		\item[$(iii)$] 
	there exists $\mu \in C(\overline{\Omega})$ such that
	\begin{align*}&\mu(z)\in \left( (r_+-p_-)\frac{N}{p_-},p^\ast(z)\right) \quad \mbox{for all } z \in \overline{\Omega}, \\
		&0<\widehat{\eta}_0 \leq \liminf_{x \to \pm \infty} \dfrac{f(z,x)x-p_+F(z,x)}{x^{\mu(z)}}\mbox{ uniformly for a.a. $z \in \Omega$;}	\end{align*}
	\item[$(iv)$] there exists $\tau \in (1,2)$ such that
	\begin{align*} & \lim_{x \to 0} \frac{f(z,x)}{x}=+\infty \mbox{ uniformly for a.a. $z \in \Omega$,}\\
		& \lim_{x \to 0} \frac{f(z,x)}{|x|^{\tau-2}x}=0\mbox{ uniformly for a.a. $z \in \Omega$,}\\ &0 \leq \liminf_{x \to 0} \dfrac{\tau F(z,x)-f(z,x)x}{|x|^{p_+}}\mbox{ uniformly for a.a. $z \in \Omega$;}
	\end{align*}	
	
	\item[$(v)$] for every $\rho >0$, there exists $\widehat{\xi}_\rho>0$ such that for a.a. $z \in \Omega$, the function $x \to f(z,x) + \widehat{\xi}_\rho |x|^{p(z)-2}x$ is nondecreasing on $[-\rho,\rho]$ and for every $s >0$, we have $0 < m_s \leq f(z,x)x$ for a.a. $z \in \Omega$, all $|x|\geq s$. 
\end{itemize}

\begin{remark}
Hypotheses $H_1\,(ii),(iii)$ imply that for a.a. $z \in \Omega$, $f(z,\cdot)$ is $(p_+-1)$-superlinear. However, this superlinearity condition on $f(z,\cdot)$ is not formulated using the AR-condition which is common in the literature when dealing with superlinear problems (see, for example, Fan-Deng \cite{Ref3}, Theorem 1.3). 
Instead we use condition $H_1\,(iii)$ which incorporates in our framework superlinear nonlinearities with slower growth as $x \to \pm \infty$, which fail to satisfy the AR-condition. Consider for example the function
$$f(z,x)=\begin{cases}
[ |x|^{\theta-2} 
-1] x & \mbox{if }|x|\leq1,\\
|x|^{p_+-2}x \ln |x|+ [|x|^{p(z)-2}-1]x & \mbox{if }1<|x|,
\end{cases}$$
with $\theta \in (1,2)$. This function satisfies hypotheses $H_1$, but fails to satisfy the AR-condition. Hypothesis $H_1\,(iv)$ implies the presence of a concave term near zero.
\end{remark}

\section{Constant Sign Solutions - Multiplicity}\label{sec:3}

In this section, we show that for $\lambda>0$ small, problem \eqref{eq0} has solutions of constant sign (positive and negative solutions). First we look for positive solutions. To this end, we introduce the $C^1$-functional $\varphi_\lambda^+ : W_0^{1,p(z)}(\Omega) \to \mathbb{R}$ defined by 
$$ \varphi_\lambda^+(u)= \int_\Omega \frac{1}{p(z)}|\nabla u(z)|^{p(z)}dz + \frac{1}{2}\|\nabla u\|_2^2 - \lambda \int_\Omega F(z,u^+)dz \quad \mbox{for all } u \in W^{1,p(z)}_0(\Omega).$$

Working with $ \varphi_\lambda^+(\cdot)$, we can produce multiple positive smooth solutions when $\lambda>0$ is small.

\begin{proposition}
	\label{prop3} If hypotheses $H_0$, $H_1$ hold, then there exists $\lambda_+>0$ such that for all $\lambda \in (0,\lambda_+)$ problem \eqref{eq0} has at least two positive solutions $u_0, \widehat{u} \in {\rm int \,}C_+$, $u_0 \neq \widehat{u}$.
\end{proposition}
\begin{proof}
	On account of hypotheses $H_1 \, (i),(iv)$, we have 
	\begin{equation}
	\label{eq1}F(z,x)\leq c_1[|x|^\tau+|x|^\theta]\quad \mbox{for a.a. $z \in \Omega$, all $x \in \mathbb{R}$, with $c_1>0$, $p_+<\theta <p^\ast_-$.}
	\end{equation}
	
	Then for every $u \in W_0^{1,p(z)}(\Omega)$, we have
	$$ \varphi_\lambda^+(u)\geq \frac{1}{p_+}\rho_p(\nabla u)-\lambda c_1 [\|u\|_\tau^\tau+\|u\|_\theta^\theta]\quad \mbox{(see \eqref{eq1}).}$$
	
	If $\|u\|\leq 1$, then by Proposition \ref{prop1} and the Poincar\'e inequality, we have $\rho_p(\nabla u)\geq \|u\|^{p_+}$. Also recall that $W^{1,p(z)}_0(\Omega) \hookrightarrow L^\tau(\Omega)$ and $W^{1,p(z)}_0(\Omega)\hookrightarrow L^\theta(\Omega)$ continuously. Therefore for $u \in W^{1,p(z)}_0(\Omega)$ with $\|u\|\leq 1$, we have
		\begin{equation}
	\label{eq2} \varphi_\lambda^+(u)\geq \frac{1}{p_+}\|u\|^{p_+}-\lambda c_2 [\|u\|^\tau+\|u\|^\theta]\quad \mbox{for some $c_2>0$.}	\end{equation}
	Let $\alpha \in \left(0, \frac{1}{p_+ -\tau}\right)$ and consider $\|u\|=\lambda^\alpha$ with $0<\lambda \leq 1$. Then from \eqref{eq2} we have
	\begin{align}
\nonumber \varphi_\lambda^+(u)	& \geq \frac{1}{p_+}\lambda^{\alpha p_+}- c_2 [\lambda^{1+\alpha \tau}+\lambda^{1+\alpha \theta}]\\ \label{eq3} & = \left[\frac{1}{p_+}-c_2 \left( \lambda^{1-\alpha (p_+-\tau)}+\lambda^{1+\alpha (\theta-p_+)}  \right)\right] \lambda^{\alpha p_+}.
	\end{align}
	
	The choice of $\alpha>0$ and since $\theta >p_+$, imply that
	$$\xi(\lambda)=c_2 \left[\lambda^{1-\alpha (p_+-\tau)}+\lambda^{1+\alpha (\theta-p_+)} \right]\to 0^+ \quad \mbox{as $\lambda \to 0^+$.}$$
	
Hence we can find $\lambda_+ \in (0,1]$ such that
		$$\xi(\lambda)<\frac{1}{p_+} \quad \mbox{for all $0<\lambda < \lambda_+$}.$$
		
Then from \eqref{eq3} we see that
\begin{equation}
\label{eq4} \varphi_\lambda^+(u)\geq m_\lambda>0 \quad \mbox{for all $\|u\|=\lambda^\alpha$, all $\lambda \in (0,\lambda_+)$.}
\end{equation}		

Let $\widehat{\lambda}_1(2)>0$ denote the principal eigenvalue of the Dirichlet Laplacian and $\widehat{u}_1(2)$ the corresponding positive, $L^2$-normalized (that is, $\|\widehat{u}_1(2)\|_2=1$) eigenfunction. We know that $\widehat{u}_1(2) \in {\rm int \, } C_+$ (see for example, Gasi\'nski-Papageorgiou \cite{Ref6}, p. 739). On account of hypothesis $H_1\, (iv)$, given $\eta>\frac{\widehat{\lambda}_1(2)}{\lambda}$, we can find $\delta >0$ such that 
\begin{equation}
\label{eq5}F(z,x)\geq \frac{\eta}{2}x^2 \quad \mbox{for a.a. $z \in \Omega$, all $|x| \leq \delta$.}
\end{equation}

Since $\widehat{u}_1(2) \in {\rm int \, } C_+$, we can find $t \in (0,1)$ small such that $0 \leq t \widehat{u}_1(2)(z)\leq \delta$ for all $z \in \overline{\Omega}$. Then
\begin{align}
\label{eq6}\varphi_\lambda^+(t \widehat{u}_1(2)) & \leq \frac{t^{p_-}}{p_-}\rho_p(\nabla  \widehat{u}_1(2))+\frac{t^2}{2}[ \widehat{\lambda}_1(2)-\lambda \eta]\\ \nonumber & \hskip 1cm \mbox{(see \eqref{eq5} and recall that $\| \widehat{u}_1(2)\|_2=1$).}
\end{align}

Note that 
$$\int_\Omega [\lambda \eta-\widehat{\lambda}_1(2)] \widehat{u}_1(2)^2dz>0.$$

Therefore from \eqref{eq6}, we have
$$ \varphi_\lambda^+(t \widehat{u}_1(2))\leq c_3t^{p_-}-c_4t^2 \quad \mbox{for some $c_3,c_4>0$.}$$

Since $2<p_-$ (see hypothesis $H_0$), choosing $t \in (0,1)$ even smaller if necessary, we have
\begin{equation}
\label{eq7}\varphi_\lambda^+(t \widehat{u}_1(2))<0 \mbox{ and } \|t \widehat{u}_1(2)\| \leq \lambda^\alpha.
\end{equation}

Using the anisotropic Sobolev embedding theorem (see Section \ref{sec:2}), we infer that $\varphi_\lambda^+(\cdot)$ is sequentially weakly lower semicontinuous. The ball $\overline{B}_{\lambda^\alpha}=\{u \in W_0^{1,p(z)}(\Omega)\, : \, \|u\|\leq \lambda^\alpha \}$ is sequentially weakly compact (recall that $W_0^{1,p(z)}(\Omega)$ is a reflexive Banach space and use the Eberlein-\v{S}mulian theorem). So, by the Weierstrass-Tonelli theorem, we can find $u_0 \in \overline{B}_{\lambda^\alpha}$ such that
\begin{equation}
\label{eq8} \varphi_\lambda^+(u_0)=\min \left[\varphi_\lambda^+(u) \, :\, u \in \overline{B}_{\lambda^\alpha}\right].
\end{equation}

From \eqref{eq7} and \eqref{eq8}, it follows that
\begin{align*}
& \varphi_\lambda^+(u_0)<0=\varphi_\lambda^+(0), \\ \Rightarrow \quad & u_0 \neq 0.
\end{align*}
Moreover, from \eqref{eq4} and \eqref{eq8}, we infer that
\begin{equation}
\label{eq9} 0<\|u_0\| < \lambda^\alpha.
\end{equation}

  From \eqref{eq9} we see that  $u_0$ is an interior point in $\overline{B}_{\lambda^\alpha}$ and a minimizer of $\varphi_\lambda^+$. Hence 
\begin{align}
\nonumber & (\varphi_\lambda^+)^\prime (u_0)=0,\\ 
\Rightarrow \quad \label{eq10} & \langle A_{p(z)}(u_0),h\rangle +\langle A_{2}(u_0),h\rangle = \lambda \int_\Omega f(z,u_0^+)h dz \quad \mbox{for all $h \in W_0^{1,p(z)}(\Omega)$.}
\end{align}

In \eqref{eq10} we choose $h=-u_0^- \in W_0^{1,p(z)}(\Omega)$ and obtain
\begin{align*}
& \rho_p(\nabla u_0^-)+\|\nabla u_0^-\|_2^2=0, \\ \Rightarrow \quad & u_0 \geq 0, \, u_0 \neq 0.
\end{align*}

From \eqref{eq10} we have that $u_0$ is a positive solution of problem \eqref{eq0} with $0<\lambda < \lambda_+$. From Fan-Zhao \cite[Theorem 4.1]{Ref4} (see also Gasi\'nski-Papageorgiou \cite[Proposition 3.1]{Ref7}), we have that $u_0 \in L^\infty(\Omega)$. Then from Tan-Fang \cite[Corollary 3.1]{Ref21} (see also Fukagai-Narukawa \cite[Lemma 3.3]{Ref5}), we have that $u_0 \in C_+ \setminus \{0\}$. Finally, the anisotropic maximum principle of Zhang \cite{Ref22} implies that $u_0 \in {\rm int \, }C_+$.

Now let $\lambda \in (0,\lambda_+)$ and consider $0 <\gamma<\lambda$. From the previous analysis, we know that problem $(P_\gamma)$ has a positive solution $u_\gamma \in {\rm int \, }C_+$. We will show that we can have
\begin{equation}
\label{eq11}u_0-u_\gamma \in {\rm int \, }C_+.
\end{equation}

First we will show that we can have a solution $u_\gamma$ of $(P_\gamma)$ such that $u_\gamma \leq u_0$. To this end let
\begin{equation}
\label{eq12}g_+(z,x)=\begin{cases}f(z,x^+) & \mbox{if }x\leq u_0(z),\\ f(z,u_0(z)) & \mbox{if }u_0(z)<x.\end{cases}
\end{equation}

This is a Carath\'eodory function. We set $G_+(z,x)=\int_0^x g_+(z,s)ds$ and consider the $C^1$-functional $\psi^+_\gamma : W_0^{1,p(z)}(\Omega) \to \mathbb{R}$ defined by
\begin{align*}
 \psi^+_\gamma(u) & = \int_\Omega \frac{1}{p(z)}|\nabla u(z)|^{p(z)}dz + \frac{1}{2}\|\nabla u\|_2^2 - \gamma \int_\Omega G_+(z,u)dz\\ & \geq \frac{1}{p_+}\rho_p(\nabla u)+ \frac{1}{2}\|\nabla u\|_2^2 - \gamma \int_\Omega G_+(z,u)dz \quad \mbox{for all } u \in W^{1,p(z)}_0(\Omega).
 \end{align*}
 
 From Proposition \ref{prop1} and \eqref{eq12}, we see that $\psi^+_\gamma(\cdot)$ is coercive. Also, it is sequentially weakly lower semicontinuous. So, we can find $u_\gamma \in W^{1,p(z)}_0(\Omega)$ such that
 \begin{equation}
 \label{eq13} \psi^+_\gamma(u_\gamma)=\min \left[\psi^+_\gamma(u) \, :\, u \in W^{1,p(z)}_0(\Omega)\right].
 \end{equation}
 
 As before, using hypothesis $H_1 \, (iv)$ and choosing $t \in (0,1)$ small so that we also have $0 \leq t \widehat{u}_1(2)\leq u_0$ (see Papageorgiou-R\u{a}dulescu-Repov\v{s} \cite{Ref13}, Proposition 4.1.22, p. 274 and recall that $u_0 \in {\rm int \, }C_+$), we will have
 \begin{align*}
 & \psi^+_\gamma(t \widehat{u}_1(2))<0,\\ \Rightarrow \quad & \psi^+_\gamma(u_\gamma)<0=\psi^+_\gamma(0) \quad \mbox{(see \eqref{eq13}),}\\ \Rightarrow \quad & u_\gamma \neq 0.
 \end{align*}
 
 From \eqref{eq13} we have 
 \begin{align}
 \nonumber & (\psi^+_\gamma)^\prime (u_\gamma)=0,\\ 
 \Rightarrow \quad \label{eq14} & \langle A_{p(z)}(u_\gamma),h\rangle +\langle A_{2}(u_\gamma),h\rangle = \gamma \int_\Omega  g_+(z,u_\gamma)h dz \quad \mbox{for all $h \in W_0^{1,p(z)}(\Omega)$.}
 \end{align}
 
 In \eqref{eq14} we choose $h=-u_\gamma^- \in W_0^{1,p(z)}(\Omega)$ and have
 \begin{align*}
 & \rho_p(\nabla u_\gamma^-)+\|\nabla u_\gamma^-\|_2^2=0, \\ \Rightarrow \quad & u_\gamma \geq 0, \, u_\gamma \neq 0.
 \end{align*}
 
 Next in \eqref{eq14} we choose $h=(u_\gamma-u_0)^+ \in W_0^{1,p(z)}(\Omega)$. We have
 \begin{align*}
 & \langle A_{p(z)}(u_\gamma),(u_\gamma-u_0)^+\rangle +\langle A_{2}(u_\gamma),(u_\gamma-u_0)^+\rangle \\  & \leq \lambda \int_\Omega f(z,u_0)(u_\gamma-u_0)^+dz \quad \mbox{(since $\gamma < \lambda$)}\\ & = \langle A_{p(z)}(u_0),(u_\gamma-u_0)^+\rangle +\langle A_{2}(u_0),(u_\gamma-u_0)^+\rangle, \\ \Rightarrow \quad & u_\gamma \leq u_0.
 \end{align*}
 
 So, we have proved that 
 \begin{equation}
 \label{eq15} u_\gamma \in [0,u_0], \, u_\gamma \neq 0.
 \end{equation}
 
 As before, from the anisotropic regularity theory and the anisotropic maximum principle, imply that $u_\gamma \in {\rm int \, }C_+$. So, we have produced a solution $u_\gamma \in {\rm int \, }C_+$ of $(P_\gamma)$ such that $u_\gamma \leq u_0$ (see \eqref{eq15}).
 
 Now, let $\rho=\|u_0\|_\infty$ and let $\widehat{\xi}_\rho>0$ be as postulated by hypothesis $H_1\, (v)$. We have
 \begin{align}
\nonumber  & - \Delta_{p(z)}u_\gamma -\Delta u_\gamma +\gamma \widehat{\xi}_\rho u_\gamma^{p(z)-1}
\\ \nonumber =  & \, \gamma \left[f(z,u_\gamma)+  \widehat{\xi}_\rho u_\gamma^{p(z)-1}\right]
\\ \nonumber \leq   & \, \gamma \left[f(z,u_0)+  \widehat{\xi}_\rho u_0^{p(z)-1}\right]\quad \mbox{(see \eqref{eq15} and hypothesis $H_1\, (v)$)}
\\ \nonumber =   & \, \lambda  f(z,u_0)+  \gamma \widehat{\xi}_\rho u_0^{p(z)-1}- (\lambda-\gamma)f(z,u_0) \\ \label{eq16} \leq & \, - \Delta_{p(z)}u_0 -\Delta u_0 +\gamma \widehat{\xi}_\rho u_0^{p(z)-1}\quad \mbox{(since $\gamma < \lambda$).}
  \end{align}
  
  Recall that $u_0 \in {\rm int \, }C_+$. So, on account of hypothesis $H_1\, (v)$, we have 
  $$0 \preceq (\lambda-\gamma)f(\cdot, u_0(\cdot)).$$
  
  Then from \eqref{eq16} and Proposition 2.4 of Papageorgiou-R\u{a}dulescu-Repov\v{s} \cite{Ref12}, we infer that \eqref{eq11} is true. 
  
  Using $u_\gamma \in {\rm int \, }C_+$ , we introduce the following truncation of $f(z,\cdot)$
  \begin{equation}
  \label{eq17} k_+(z,x)=\begin{cases} f(z,u_\gamma(z)) & \mbox{if }x \leq u_\gamma(z),\\
  f(z,x) & \mbox{if }u_\gamma(z)<x.
  \end{cases}
  \end{equation}
  
  We set $K_+(z,x)=\int_0^x k_+(z,s)ds$ and consider the $C^1$-functional $\widehat{\varphi}^+_\lambda : W_0^{1,p(z)}(\Omega) \to \mathbb{R}$ defined by
  $$
  \widehat{\varphi}^+_\lambda(u)  = \int_\Omega \frac{1}{p(z)}|\nabla u(z)|^{p(z)}dz + \frac{1}{2}\|\nabla u\|_2^2 - \lambda \int_\Omega K_+(z,u)dz\quad \mbox{for all } u \in W^{1,p(z)}_0(\Omega).
$$

From \eqref{eq17} we see that
\begin{equation}
\label{eq18}  \varphi^+_\lambda \Big|_{[u_\gamma)}= \widehat{\varphi}^+_\lambda\Big|_{[u_\gamma)}+\widehat{\beta}_\lambda \quad \mbox{with $\widehat{\beta}_\lambda \in \mathbb{R}$.}
\end{equation}

From the first part of the proof, we know that $u_0 \in {\rm int \, }C_+$ is a local minimizer of $\varphi^+_\lambda$. Then \eqref{eq11} and \eqref{eq18} imply that 
\begin{align}
\nonumber & \mbox{$u_0$ is a local $C_0^1(\overline{\Omega})$-minimizer of $\widehat{\varphi}^+_\lambda(\cdot)$,}\\ \label{eq19} \Rightarrow \quad  & \mbox{$u_0$ is a local $W_0^{1,p(z)}(\Omega)$-minimizer of $\widehat{\varphi}^+_\lambda(\cdot)$}
\end{align}
(see Tan-Fang \cite{Ref21}, Theorem 3.2 and Gasi\'nski-Papageorgiou \cite{Ref7}, Proposition 3.3). Using \eqref{eq17}, we can easily check that 
\begin{equation}
\label{eq20} K_{\widehat{\varphi}^+_\lambda}\subseteq [u_\gamma) \cap {\rm int \, }C_+.
\end{equation}

This implies that we may assume that
\begin{equation}
\label{eq21} K_{\widehat{\varphi}^+_\lambda} \mbox{ is finite}
\end{equation}
(otherwise we already have a whole sequence of distinct positive smooth solutions of \eqref{eq0} and so we are done). Then \eqref{eq21}, \eqref{eq19} and Theorem 5.7.6, p. 449, of Papageorgiou-R\u{a}dulescu-Repov\v{s} \cite{Ref13}, imply that we can find $\rho \in (0,1)$ small such that
\begin{equation}
\label{eq22}\widehat{\varphi}^+_\lambda(u_0)< \inf \left[\widehat{\varphi}^+_\lambda(u)\, : \, \|u-u_0\|=\rho\right]=m_\lambda^+.
\end{equation}

If $u \in {\rm int \, }C_+$, then from hypothesis $H_1\, (ii)$ we have 
\begin{equation}
\label{eq23}\widehat{\varphi}^+_\lambda(tu) \to -\infty \quad \mbox{as $t \to +\infty$.}
\end{equation}

Moreover, \eqref{eq18} and Proposition 4.1 of Gasi\'nski-Papageorgiou \cite{Ref7}, implies that 
\begin{equation}
\label{eq24}\widehat{\varphi}^+_\lambda(\cdot) \mbox{ satisfies the $C$-condition (see hypothesis $H_1 \, (iii)$).}
\end{equation}

From \eqref{eq22}, \eqref{eq23} and \eqref{eq24}, we see that we can use the mountain pass theorem and obtain $\widehat{u} \in W_0^{1,p(z)}(\Omega)$ such that 
\begin{equation}
\label{eq25}\begin{cases}\widehat{u} \in K_{\widehat{\varphi}^+_\lambda}\subseteq [u_\gamma) \cap {\rm int \, }C_+ & \mbox{(see \eqref{eq20})},\\ \widehat{\varphi}^+_\lambda(u_0)<  m_\lambda^+ \leq \widehat{\varphi}^+_\lambda(\widehat{u})& \mbox{(see \eqref{eq22})}.\end{cases}
\end{equation}

From \eqref{eq25} and \eqref{eq17} it follows that $\widehat{u} \in {\rm int \,}C_+$ is a positive solution of problem \eqref{eq0} ($\lambda \in (0,\lambda_+)$), $\widehat{u} \neq u_0$. 
	\end{proof}

In a similar fashion we can generate two negative smooth solutions when $\lambda>0$ is small. In this case we start with the $C^1$-functional $\varphi_\lambda^- : W_0^{1,p(z)}(\Omega) \to \mathbb{R}$ defined by 
$$ \varphi_\lambda^-(u)= \int_\Omega \frac{1}{p(z)}|\nabla u(z)|^{p(z)}dz + \frac{1}{2}\|\nabla u\|_2^2 - \lambda \int_\Omega F(z,-u^-)dz \quad \mbox{for all } u \in W^{1,p(z)}_0(\Omega).$$

Using this functional and reasoning as in the ``positive'' case, we have the following multiplicity result.

\begin{proposition}
	\label{prop4} If hypotheses $H_0$, $H_1$ hold, then there exists $\lambda_->0$ such that for all $\lambda \in (0,\lambda_-)$ problem \eqref{eq0} has at least two  negative solutions $v_0, \widehat{v} \in -{\rm int \,}C_+$, $v_0 \neq \widehat{v}$.
\end{proposition}

\section{Extremal Constant Sign Solutions}\label{sec:4}

Let $S_\lambda^+$ be the set of positive solutions of \eqref{eq0} and $S_\lambda^-$ be the set of negative solutions of \eqref{eq0}. We know that:
\begin{align*}
& \emptyset \neq S_\lambda^+ \subseteq {\rm int \, } C_+ \quad \mbox{for all $\lambda \in (0,\lambda_+)$ (see Proposition \ref{prop3}),}\\	
& \emptyset \neq S_\lambda^- \subseteq -{\rm int \, } C_+ \quad \mbox{for all $\lambda \in (0,\lambda_-)$ (see Proposition \ref{prop4}).}	
\end{align*}	

In this section we show that $S_\lambda^+$ has a smallest element $\overline{u}_\lambda \in {\rm int \, } C_+$, that is, $\overline{u}_\lambda \leq u$ for all $u \in S_\lambda^+$ and $S_\lambda^-$ has a biggest element $\overline{v}_\lambda \in -{\rm int \, } C_+$, that is, $v \leq \overline{v}_\lambda$ for all $v \in S_\lambda^-$. We call $\overline{u}_\lambda$ and $\overline{v}_\lambda$ the ``extremal'' constant sign solutions of \eqref{eq0}. In Section \ref{sec:5} these solutions will be used to produce a nodal (sign-changing) solution of \eqref{eq0}. Indeed, if we can produce a nontrivial solution of \eqref{eq0} in the order interval $[\overline{v}_\lambda,\overline{u}_\lambda]$ distinct from $\overline{u}_\lambda$ and $\overline{v}_\lambda$, on account of the extremality of $\overline{u}_\lambda$ and $\overline{v}_\lambda$, this solution will be nodal.

To produce the extremal constant sign solutions, we need some preparation. Let $\lambda \in (0,\lambda_+)$ and let $\eta >\frac{\widehat{\lambda}_1(2)}{\lambda}$. On account of hypotheses $H_1 \, (i),(iv)$, we can find $c_5>0$ such that
\begin{equation}
\label{eq26} f(z,x)x \geq \eta x^2-c_5|x|^{r_+} \quad \mbox{for a.a. $z \in \Omega$, all $x \in \mathbb{R}$.}
\end{equation}

This unilateral growth restriction on $f(z,\cdot)$, leads to the following auxiliary anisotropic $(p,2)$-problem
\begin{equation}\label{eq0bis}\tag{$Q_\lambda$} 
\begin{cases}-\Delta_{p(z)} u(z)-\Delta u(z)=\lambda \left[\eta u(z)-c_5|u(z)|^{r_+-2}u(z)\right]  & \mbox{in } \Omega,\\ u \Big|_{\partial \Omega} =0, \, \lambda>0, \, u>0.& \end{cases}
\end{equation}

For this problem we have the following result

\begin{proposition}
	\label{prop5} If hypotheses $H_0$ hold, then for every $\lambda>0$ problem \eqref{eq0bis} has a unique positive solution $u_\lambda^\ast \in {\rm int \, }C_+$, and since problem \eqref{eq0bis} is odd, $v_\lambda^\ast=-u_\lambda^\ast \in -{\rm int \, }C_+$ is the unique negative solution.
\end{proposition}

\begin{proof}
Consider the $C^1$-functional $\sigma^+_\lambda : W_0^{1,p(z)}(\Omega) \to \mathbb{R}$ defined by
	\begin{align*}
		\sigma^+_\lambda(u) & = \int_\Omega \frac{1}{p(z)}|\nabla u(z)|^{p(z)}dz + \frac{1}{2}\|\nabla u\|_2^2 + \frac{\lambda c_5}{r_+}\|u^+\|_{r_+}^{r_+}-\frac{\lambda}{2}\eta\|u^+\|_2^2\\ & \geq \frac{1}{p_+}\rho_p(\nabla u)+ \frac{1}{2}\|\nabla u\|_2^2 + \frac{\lambda c_5}{r_+}\|u^+\|_{r_+}^{r_+}-\frac{\lambda}{2}\eta \|u^+\|_2^2 \quad \mbox{for all } u \in W^{1,p(z)}_0(\Omega).
	\end{align*}
	
Since $p_->2$, from this last inequality we infer that $\sigma^+_\lambda(\cdot)$ is coercive. Also, it is sequentially weakly lower semicontinuous. So, we can find $u^\ast_\lambda \in W^{1,p(z)}_0(\Omega)$ such that
	\begin{equation}
	\label{eq27} \sigma^+_\lambda(u^\ast_\lambda)=\min \left[\sigma^+_\lambda(u) \, :\, u \in W^{1,p(z)}_0(\Omega)\right].
	\end{equation}
	
Let $t \in (0,1)$. We have
$$
\sigma_\lambda^+(t \widehat{u}_1(2))  \leq \frac{t^{p_-}}{p_-}\rho_p(\nabla  \widehat{u}_1(2))+\frac{t^2}{2}\left[\int_\Omega \widehat{\lambda}_1(2)-\lambda \eta\right]\widehat{u}_1(2)^2dz + \frac{\lambda t^{r_+}}{r_+}\|\widehat{u}_1(2) \|^{r_+}_{r_+}.$$

From the choice of $\eta$ we see that
$$\beta_0=\int_\Omega (\lambda \eta- \widehat{\lambda}_1(2))\widehat{u}_1(2)^2dz>0.$$

Therefore we can write that
$$ \sigma_\lambda^+(t \widehat{u}_1(2)) \leq  c_6 t^{p_-}- c_7 t^{2}\quad \mbox{for some $c_6,c_7>0$ (recall that $p_-<r_+$).}$$

Since $2<p_-$, taking $t \in (0,1)$ even smaller if necessary, we have
 \begin{align*}
	& \sigma^+_\lambda(t \widehat{u}_1(2))<0,\\ \Rightarrow \quad &  \sigma^+_\lambda(u^\ast_\lambda)<0= \sigma^+_\lambda(0) \quad \mbox{(see \eqref{eq27}),}\\ \Rightarrow \quad & u^\ast_\lambda \neq 0.
\end{align*}

From \eqref{eq27} we have 
\begin{align}
	\nonumber & (\sigma^+_\lambda)^\prime (u^\ast_\lambda)=0,\\ 
	\Rightarrow \quad \label{eq28} & \langle A_{p(z)}(u^\ast_\lambda),h\rangle +\langle A_{2}(u^\ast_\lambda),h\rangle =  \lambda \int_\Omega \left[\eta (u^\ast_\lambda)^+-c_5((u^\ast_\lambda)^+)^{r_+-1} \right]h dz \\ \nonumber & \hskip 8cm \mbox{for all $h \in W_0^{1,p(z)}(\Omega)$.}
\end{align}

In \eqref{eq28} we choose $h=-(u^\ast_\lambda)^- \in W_0^{1,p(z)}(\Omega)$ and obtain
\begin{align*}
	& \rho_p(\nabla(u^\ast_\lambda)^-)+\|\nabla (u^\ast_\lambda)^-\|_2^2=0, \\ \Rightarrow \quad & u^\ast_\lambda \geq 0, \, u^\ast_\lambda \neq 0.
\end{align*}

Then from \eqref{eq27} we see that $u^\ast_\lambda$ is a positive solution of \eqref{eq0bis}. As before (see the proof of Proposition \ref{prop3}), the anisotropic regularity theory and the anisotropic maximum principle, imply that 
$$ u^\ast_\lambda \in {\rm int \, }C_+.$$

Next we show the uniqueness of this positive solution. To this end we consider the integral functional $j: L^1(\Omega)\to \overline{\mathbb{R}}=\mathbb{R} \cup \{+\infty\}$ defined by
$$j(u)= \begin{cases} \int_\Omega \frac{1}{p(z)}|\nabla u^{1/2}|^{p(z)}dz+\frac{1}{2}\|\nabla u^{1/2}\|_2^2 & \mbox{if }u \geq 0, \, u^{1/2} \in W_0^{1,p(z)}(\Omega),\\ +\infty & \mbox{otherwise.} 
\end{cases}
$$

From Theorem 2.2 Tak\'a\u{c}-Giacomoni \cite{Ref20}, we know that $j(\cdot)$ is convex. Let ${\rm dom \,}j=\{u \in L^1(\Omega) \, : \, j(u)< \infty \}$ (the effective domain of $j(\cdot)$) and suppose $\widehat{u}^\ast_\lambda$ is another positive solution of \eqref{eq0bis}. Again we have that $\widehat{u}^\ast_\lambda \in {\rm int \, }C_+$. Hence using Proposition 4.1.22, p. 274, of Papageorgiou-R\u{a}dulescu-Repov\v{s} \cite{Ref13}, we have 
$$ \frac{\widehat{u}^\ast_\lambda}{u^\ast_\lambda}\in L^\infty(\Omega) \mbox{ and } \frac{u^\ast_\lambda}{\widehat{u}^\ast_\lambda}\in L^\infty(\Omega).$$ 

Let $h= (u^\ast_\lambda)^2-(\widehat{u}^\ast_\lambda)^2$. Then for $|t|<1$ small, we have
$$ (u^\ast_\lambda)^2 + th \in {\rm dom \,}j, \quad (\widehat{u}^\ast_\lambda)^2 + th \in {\rm dom \,}j.$$

Thus the convexity of $j(\cdot)$ implies the Gateaux differentiability of $j(\cdot)$ at  $(u^\ast_\lambda)^2$ and at $(\widehat{u}^\ast_\lambda)^2$ in the direction $h$. Moreover, a direct calculation using Green's identity (see also \cite{Ref20}, Theorem 2.5), gives
\begin{align*}
	j^\prime((u^\ast_\lambda)^2)(h) & = \frac{1}{2}\int_\Omega \frac{-\Delta_{p(z)} u^\ast_\lambda -\Delta u^\ast_\lambda }{u^\ast_\lambda}h dz\\
			&  = \frac{\lambda}{2}\int_\Omega [\eta -c_5 (u^\ast_\lambda)^{r_+-2}]h dz,\\
	j^\prime((\widehat{u}^\ast_\lambda)^2)(h) & = \frac{1}{2}\int_\Omega \frac{-\Delta_{p(z)} \widehat{u}^\ast_\lambda -\Delta \widehat{u}^\ast_\lambda }{\widehat{u}^\ast_\lambda}h dz\\
		& = \frac{\lambda}{2}\int_\Omega [\eta -c_5 (\widehat{u}^\ast_\lambda)^{r_+-2}] hdz.
	\end{align*}

The convexity of $j(\cdot)$ implies the monotonicity of $j^\prime(\cdot)$. Hence 
\begin{align*}
& 0 \leq c_5 \int_\Omega  \left[(\widehat{u}^\ast_\lambda)^{r_+-2}-(u^\ast_\lambda)^{r_+-2}\right]	  ((u^\ast_\lambda)^2-(\widehat{u}^\ast_\lambda)^2)dz \leq 0,\\ \Rightarrow \quad & u_\lambda^\ast =\widehat{u}_\lambda^\ast.
\end{align*}	

This proves the uniqueness of the positive solution $u_\lambda^\ast \in {\rm int \,}C_+$ of \eqref{eq0bis}. Since the equation is odd, it follows that $v_\lambda^\ast=-u_\lambda^\ast \in - {\rm int \, }C_+$ is the unique negative solution of \eqref{eq0bis}, $\lambda>0$. 
	\end{proof}

The solution $u_\lambda^\ast$ (resp. $v_\lambda^\ast$), will provide a lower bound (resp. an upper bound) for the solution set $S_\lambda^+$ (resp. $S_\lambda^-$). These bounds are important in generating the extremal constant sign solutions. 

\begin{proposition}
	\label{prop6} If hypotheses $H_0$, $H_1$ hold, then $u_\lambda^\ast \leq u$ for all $u \in S_\lambda^+$ and $v \leq v_\lambda^\ast$ for all $v \in S_\lambda^-$.
\end{proposition}

\begin{proof}
	 Let $u \in S_\lambda^+ \subseteq {\rm int \, }C_+$ and consider the Carath\'eodory function $e: \Omega \times \mathbb{R} \to \mathbb{R}$ defined by
	\begin{equation}
	\label{eq29} e(z,x)=\begin{cases} \eta x^+-c_5(x^+)^{r_+-1} & \mbox{if }x \leq u(z),\\
	\eta u(z)-c_5u(z)^{r_+-1} & \mbox{if }u(z)<x.
	\end{cases}
	\end{equation}
	
	We set $E(z,x)=\int_0^xe(z,s)ds$ and consider the $C^1$-functional $\widehat{\sigma}^+_\lambda : W_0^{1,p(z)}(\Omega) \to \mathbb{R}$ defined by
\begin{align*}
	\widehat{\sigma}^+_\lambda(u) & = \int_\Omega \frac{1}{p(z)}|\nabla u(z)|^{p(z)}dz + \frac{1}{2}\|\nabla u\|_2^2 - \lambda \int_\Omega E(z,u)dz\\ & \geq \frac{1}{p_+}\rho_p(\nabla u)+\frac{1}{2}\|\nabla u\|_2^2 -c_8 \quad \mbox{for some $c_8>0$ (see \eqref{eq29}), all } u \in W^{1,p(z)}_0(\Omega).
\end{align*}

It follows that $\widehat{\sigma}^+_\lambda(\cdot)$ is coercive. Also, it is sequentially weakly lower semicontinuous. So, we can find $\widetilde{u}^\ast_\lambda \in W^{1,p(z)}_0(\Omega)$ such that
\begin{equation}
\label{eq30} \widehat{\sigma}^+_\lambda(\widetilde{u}^\ast_\lambda)=\min \left[\widehat{\sigma}^+_\lambda(u) \, :\, u \in W^{1,p(z)}_0(\Omega)\right].
\end{equation}

As before (see the proof of Proposition \ref{prop5}), for  $t \in (0,1)$ small, we will have
\begin{align*}
& 0 \leq t \widehat{u}_1(2)\leq u \mbox{ and }\widehat{\sigma}_\lambda^+(t \widehat{u}_1(2)) <0\\ &\mbox{(recall that $u \in {\rm int \, }C_+$, see \eqref{eq29} and recall $2< p_-<r_+$)}	\\ \Rightarrow \quad &  \widehat{\sigma}^+_\lambda(\widetilde{u}^\ast_\lambda)<0= \widehat{\sigma}^+_\lambda(0) \quad \mbox{(see \eqref{eq30}),}\\ \Rightarrow \quad & \widetilde{u}^\ast_\lambda \neq 0.
\end{align*}

From \eqref{eq30} we have 
\begin{align}
	\nonumber & (\widehat{\sigma}^+_\lambda)^\prime (\widetilde{u}^\ast_\lambda)=0,\\ 
	\Rightarrow \quad \label{eq31} & \langle A_{p(z)}(\widetilde{u}^\ast_\lambda),h\rangle +\langle A_{2}(\widetilde{u}^\ast_\lambda),h\rangle =  \lambda \int_\Omega e(z,\widetilde{u}^\ast_\lambda)h dz \quad \mbox{for all $h \in W_0^{1,p(z)}(\Omega)$.}
\end{align}

In \eqref{eq31} first we choose $h=-(\widetilde{u}^\ast_\lambda)^- \in W_0^{1,p(z)}(\Omega)$. We obtain
\begin{align*}
	& \rho_p(\nabla(\widetilde{u}^\ast_\lambda)^-)+\|\nabla (\widetilde{u}^\ast_\lambda)^-\|_2^2=0 \quad \mbox{(see \eqref{eq29})}, \\ \Rightarrow \quad & \widetilde{u}^\ast_\lambda \geq 0, \, \widetilde{u}^\ast_\lambda \neq 0.
\end{align*}

Next in \eqref{eq31} we choose $h=(\widetilde{u}^\ast_\lambda-u)^+ \in W_0^{1,p(z)}(\Omega)$. We have
\begin{align*}
	&A_{p(z)}(\widetilde{u}^\ast_\lambda),(\widetilde{u}^\ast_\lambda-u)^+\rangle +\langle A_{2}(\widetilde{u}^\ast_\lambda),(\widetilde{u}^\ast_\lambda-u)^+\rangle\\ & = \lambda \int_\Omega \left[\eta u-c_5u^{r_+-1} \right](\widetilde{u}^\ast_\lambda-u)^+ dz\quad \mbox{(see \eqref{eq29})} \\ & \leq  \lambda \int_\Omega f(z,u)(\widetilde{u}^\ast_\lambda-u)^+ dz\quad \mbox{(see \eqref{eq26})} \\	& = A_{p(z)}(u),(\widetilde{u}^\ast_\lambda-u)^+\rangle +\langle A_{2}(u),(\widetilde{u}^\ast_\lambda-u)^+\rangle, \\ \Rightarrow \quad & \widetilde{u}^\ast_\lambda \leq u.
\end{align*}

So, we have proved that
\begin{equation}
\label{eq32}\widetilde{u}^\ast_\lambda \in [0,u], \quad \widetilde{u}^\ast_\lambda \neq 0.
\end{equation}

From \eqref{eq32}, \eqref{eq29} and \eqref{eq31}, we see that $\widetilde{u}^\ast_\lambda$ is a positive solution of \eqref{eq0bis}. Then Proposition \ref{prop5} implies that 
\begin{align*}
	& \widetilde{u}^\ast_\lambda=u^\ast_\lambda,\\ \Rightarrow \quad &
	u^\ast_\lambda \leq u \quad \mbox{for all $u \in S_\lambda^+$.}
\end{align*}

Similarly we show that $v \leq 	\widetilde{v}^\ast_\lambda$ for all $v \in S_\lambda^-$.
	\end{proof}

Now we are ready to produce the extremal constant sign solutions for problem \eqref{eq0}. Let $\lambda^\ast = \min \{\lambda_+,\lambda_-\}$.

\begin{proposition}
	\label{prop7} If hypotheses $H_0$, $H_1$ hold and $\lambda \in (0,\lambda^\ast)$, then problem \eqref{eq0} has a smallest positive solution $\overline{u}_\lambda \in S_\lambda^+ \subseteq {\rm int \, }C_+$ and a biggest negative solution  $\overline{v}_\lambda  \in S_\lambda^- \subseteq - {\rm int \, }C_+$.
\end{proposition}

\begin{proof}
From Papageorgiou-R\u{a}dulescu-Repov\v{s} \cite{Ref11} (see the proof of Proposition \ref{prop7}), we have that $S_\lambda^+$ is downward directed (that is, if $u_1,u_2 \in S_\lambda^+$, then we can find $u \in S_\lambda^+$ such that $u \leq u_1$, $u \leq u_2$). Therefore using Lemma 3.10, p. 178, of Hu-Papageorgiou \cite{Ref8}, we can find $\{u_n\}_{n \in \mathbb{N}}\subseteq S_\lambda^+$ decreasing such that
$$ \inf S_\lambda^+ = \inf\limits_{n \in \mathbb{N}}u_n.$$

We have
\begin{align}
\label{eq33} &  \langle A_{p(z)}(u_n),h\rangle +\langle A_{2}(u_n),h\rangle =  \lambda \int_\Omega f(z,u_n)h dz\\ & \nonumber \hskip5cm \mbox{ for all $h \in W_0^{1,p(z)}(\Omega)$, all $n \in \mathbb{N}$,}\\
\label{eq34} & u_\lambda^\ast \leq u_n \leq u_1 \quad \mbox{for all $n \in \mathbb{N}$.}
\end{align}
	
	In \eqref{eq33} we choose $h = u_n \in W_0^{1,p(z)}(\Omega)$ and obtain
	\begin{equation}
	\label{eq35}\rho_p(\nabla u_n)+\|\nabla u_n\|_2^2=\lambda \int_\Omega f(z,u_n)u_n dz  \quad \mbox{for all $n \in \mathbb{N}$.}
	\end{equation}
	
	From \eqref{eq34}, \eqref{eq35} and hypothesis $H_1\, (i)$, we infer that 
	$$\{u_n\}_{n \in \mathbb{N}}\subseteq  W_0^{1,p(z)}(\Omega) \mbox{ is bounded.}$$
	
	So, we may assume that
	\begin{equation}
	\label{eq36}u_n \xrightarrow{w} \overline{u}_\lambda \mbox{ in }W_0^{1,p(z)}(\Omega) \mbox{ and } u_n \to \overline{u}_\lambda \mbox{ in }L^{r(z)}(\Omega).
	\end{equation}
	
	In \eqref{eq33} we choose $h = u_n-\overline{u}_\lambda \in W_0^{1,p(z)}(\Omega)$, pass to the limit as $n \to \infty$ and use \eqref{eq36}. We obtain
\begin{align}
	\nonumber & \lim_{n \to \infty}\left[ \langle A_{p(z)}(u_n),u_n-\overline{u}_\lambda\rangle +\langle A_{2}(u_n),u_n-\overline{u}_\lambda\rangle \right]=0,\\
	\Rightarrow \quad  & \nonumber \limsup_{n \to \infty}\left[ \langle A_{p(z)}(u_n),u_n-\overline{u}_\lambda\rangle +\langle A_{2}(\overline{u}_\lambda),u_n-\overline{u}_\lambda\rangle \right]\leq 0,\\
\nonumber & \hskip 4cm \mbox{(from the monotonicity of $A_2(\cdot)$),}\\
\Rightarrow \quad  & \nonumber \limsup_{n \to \infty}  \langle A_{p(z)}(u_n),u_n-\overline{u}_\lambda\rangle \leq 0 \quad \mbox{(see \eqref{eq36}),}\\ \label{eq37} \Rightarrow \quad  & u_n \to \overline{u}_\lambda \quad \mbox{in $W_0^{1,p(z)}(\Omega)$ (see Proposition \ref{prop2}).}
\end{align}

Therefore, if in \eqref{eq33} we pass to the limit as $n \to \infty$ and use \eqref{eq37}, we obtain
\begin{align*}
 &  \langle A_{p(z)}(\overline{u}_\lambda),h\rangle +\langle A_{2}(\overline{u}_\lambda),h\rangle =  \lambda \int_\Omega f(z,\overline{u}_\lambda)h dz\quad \mbox{for all $h \in W_0^{1,p(z)}(\Omega)$,}\\
 & u_\lambda^\ast \leq \overline{u}_\lambda, \\ \Rightarrow  \quad & \overline{u}_\lambda \in S_\lambda^+ \mbox{ and } \overline{u}_\lambda=\inf S_\lambda^+.
\end{align*}

For the negative solutions, we know that $S_\lambda^-$ is upward directed (that is, if $v_1,v_2 \in S_\lambda^-$, then we can find $v \in S_\lambda^-$ such that $v_1 \leq v$, $v_2 \leq v$). Reasoning as above, we produce $\overline{v}_\lambda \in S_\lambda^- \subseteq - {\rm int \,}C_+$ such that  $v \leq \overline{v}_\lambda$ for all $v\in S_\lambda^-$.
		\end{proof}

\section{Nodal Solutions}\label{sec:5}

In this section using the extremal constant sign solutions and following the approach outlined in the beginning of Section \ref{sec:4}, we will produce a nodal (sign-changing) solution for problem \eqref{eq0}, $\lambda \in (0,\lambda^\ast)$.

Let $\overline{u}_\lambda \in   {\rm int \,}C_+$ and $\overline{v}_\lambda \in   - {\rm int \,}C_+$ be the two extremal constant sign solutions produced in Proposition \ref{prop7}. We introduce the Carath\'eodory function $\overline{f}: \Omega \times \mathbb{R} \to \mathbb{R}$ defined by
\begin{equation}
\label{eq38} \overline{f}(z,x)=\begin{cases} f(z,\overline{v}_\lambda(z) ) & \mbox{if }x < \overline{v}_\lambda(z),\\ f(z,x) & \mbox{if } \overline{v}_\lambda(z) \leq x \leq \overline{u}_\lambda(z),\\
f(z,\overline{u}_\lambda(z) ) & \mbox{if }\overline{u}_\lambda(z)<x.
\end{cases}
\end{equation}

We also consider the positive and negative truncations of $\overline{f}(z,\cdot)$, namely the Carath\'eodory functions
\begin{equation}
\label{eq39}\overline{f}_\pm(z,x)=\overline{f}(z,\pm x^\pm).
\end{equation}

We set $\overline{F}(z,x)=\int_0^x\overline{f}(z,s)ds$ and $\overline{F}_\pm(z,x)=\int_0^x\overline{f}_\pm(z,s)ds$, and then introduce the $C^1$-functionals $\overline{\varphi}_\lambda,\overline{\varphi}^\pm_\lambda : W_0^{1,p(z)}(\Omega) \to \mathbb{R}$ defined by
\begin{align*}
 &	\overline{\varphi}_\lambda(u) = \int_\Omega \frac{1}{p(z)}|\nabla u(z)|^{p(z)}dz + \frac{1}{2}\|\nabla u\|_2^2 - \lambda \int_\Omega \overline{F}(z,u)dz,\\ & \overline{\varphi}^\pm_\lambda(u) = \int_\Omega \frac{1}{p(z)}|\nabla u(z)|^{p(z)}dz + \frac{1}{2}\|\nabla u\|_2^2 - \lambda \int_\Omega \overline{F}_\pm(z,u)dz, \\ & \hskip 7cm \mbox{for all } u \in W^{1,p(z)}_0(\Omega).
\end{align*}

Using \eqref{eq38} and \eqref{eq39} and arguing as in the proof of Proposition \ref{prop6}, since $\overline{u}_\lambda,\overline{v}_\lambda$ are the extremal constant sign solutions, we obtain the following proposition. 

\begin{proposition}
	\label{prop8} If hypotheses $H_0$, $H_1$ hold and $\lambda \in (0,\lambda^\ast)$, then $K_{\overline{\varphi}_\lambda}\subseteq [\overline{v}_\lambda,\overline{u}_\lambda] \cap C_0^1(\overline{\Omega})$, $K_{\overline{\varphi}^+_\lambda}=\{0,\overline{u}_\lambda\}$, $K_{\overline{\varphi}^-_\lambda}=\{0,\overline{v}_\lambda\}$. 
\end{proposition}

The next result will allow the use of the mountain pass theorem.

\begin{proposition}
	\label{prop9} If hypotheses $H_0$, $H_1$ hold and $\lambda \in (0,\lambda^\ast)$, then the two extremal constant sign solutions $\overline{u}_\lambda \in {\rm int \,}C_+$ and $\overline{v}_\lambda \in- {\rm int \,}C_+$ are local minimizers of $\overline{\varphi}_\lambda(\cdot)$. 
\end{proposition}

\begin{proof}
  From \eqref{eq38}, \eqref{eq39}  and hypothesis $H_1\, (i)$, we have
 $$\int_\Omega \overline{F}(z,u)dz \leq \widetilde{c} \quad \mbox{for some $\widetilde{c}>0$, all $u \in W^{1,p(z)}_0(\Omega)$.}$$
 
 Therefore
 \begin{align*}
 	& \overline{\varphi}^+_\lambda(u)\geq \frac{1}{p_+}\rho_p(\nabla u)-\lambda \widetilde{c} \quad \mbox{for all $u \in W^{1,p(z)}_0(\Omega)$,}\\
 	\Rightarrow \quad & \overline{\varphi}^+_\lambda(\cdot) \mbox{ is coercive (see Proposition \ref{prop1}).}
 	\end{align*}
 
Also, it is sequentially weakly lower semicontinuous. So, we can find $\overline{w}_\lambda \in W^{1,p(z)}_0(\Omega)$ such that
\begin{align*}
& \overline{\varphi}^+_\lambda(\overline{w}_\lambda)=\min \left[\overline{\varphi}^+_\lambda(u) \, :\, u \in W^{1,p(z)}_0(\Omega)\right]<0=\overline{\varphi}^+_\lambda(0)\\ & \hskip5cm \mbox{(see the proof of Proposition \ref{prop6}),}\\ \Rightarrow \quad & \overline{w}_\lambda\neq 0.
\end{align*}

Since $\overline{w}_\lambda \in K_{\overline{\varphi}^+_\lambda} \setminus \{0\}$, from Proposition \ref{prop8} we infer that
$$\overline{w}_\lambda=\overline{u}_\lambda \in {\rm int \,}C_+.$$

But from \eqref{eq38} and \eqref{eq39} it is clear that
$$\overline{\varphi}_\lambda \Big|_{C_+}= \overline{\varphi}^+_\lambda\Big|_{C_+}.$$

It follows that 
\begin{align*}
 & \mbox{$\overline{u}_\lambda \in {\rm int \,}C_+$ is a local $C_0^1(\overline{\Omega})$-minimizer of $\overline{\varphi}_\lambda(\cdot)$,}\\  \Rightarrow \quad  & \mbox{$\overline{u}_\lambda \in {\rm int \,}C_+$ is a local $W_0^{1,p(z)}(\Omega)$-minimizer of $\overline{\varphi}_\lambda(\cdot)$ (see \cite{Ref21}, \cite{Ref7}).}
\end{align*}

Similarly using this time $\overline{\varphi}^-_\lambda(\cdot)$ we show that $\overline{v}_\lambda \in -{\rm int \,}C_+$ is a local minimizer of $\overline{\varphi}_\lambda(\cdot)$.
\end{proof}

It is clear from Proposition \ref{prop8}, that we may assume that
\begin{equation}\label{eq40}K_{\overline{\varphi}_\lambda} \mbox{ is finite.}\end{equation}

Otherwise, we already have a sequence of distinct smooth nodal solutions of \eqref{eq0} and so we are done.

Also, we may assume that 
\begin{equation}
\label{eq41}\overline{\varphi}_\lambda(\overline{v}_\lambda)\leq \overline{\varphi}_\lambda(\overline{u}_\lambda).
\end{equation}

The reasoning is similar if the opposite inequality holds.

\begin{proposition}
	\label{prop10} If hypotheses $H_0$, $H_1$ hold and $\lambda \in (0,\lambda^\ast)$, then the problem \eqref{eq0} has a nodal solution $y_0 \in  C_0^1(\overline{\Omega})$. 
\end{proposition}

\begin{proof}
From \eqref{eq40}, \eqref{eq41} and Theorem 5.7.6, p. 449, of Papageorgiou-R\u{a}dulescu-Repov\v{s} \cite{Ref13}, we can find $\rho \in (0,1)$ small such that
	\begin{equation}
	\label{eq42}\overline{\varphi}_\lambda(\overline{v}_\lambda)\leq \overline{\varphi}_\lambda(\overline{u}_\lambda)< \inf \left[\overline{\varphi}_\lambda(u)\, : \, \|u-\overline{u}_\lambda\|=\rho\right]=m_\lambda, \quad \|\overline{v}_\lambda-\overline{u}_\lambda\|>\rho.
	\end{equation}
	
Also $\overline{\varphi}_\lambda(\cdot)$ is coercive (see \eqref{eq38}). Hence  Proposition 5.1.15, p. 369, of Papageorgiou-R\u{a}dulescu-Repov\v{s} \cite{Ref13}, implies that 
	\begin{equation}
	\label{eq43}\overline{\varphi}_\lambda(\cdot) \mbox{ satisfies the $C$-condition.}
	\end{equation}
	
	Then \eqref{eq42} and \eqref{eq43} permit the use of the mountain pass theorem. So, we can find $y_0 \in W_0^{1,p(z)}(\Omega)$ such that
	\begin{equation}
	\label{eq44}y_0 \in K_{\overline{\varphi}_\lambda}\subseteq [\overline{v}_\lambda,\overline{u}_\lambda] \cap C_0^1(\overline{\Omega}), \quad m_\lambda \leq \overline{\varphi}_\lambda(y_0).
	\end{equation}
	
	From \eqref{eq44}, \eqref{eq38} and \eqref{eq42}, we infer that
	$$y_0 \in C_0^1(\overline{\Omega}) \mbox{ is a solution of \eqref{eq0}, } y_0 \not \in \{ \overline{u}_\lambda,\overline{v}_\lambda \}.$$
	
From Theorem 6.5.8, p. 527, of Papageorgiou-R\u{a}dulescu-Repov\v{s} \cite{Ref13}, we know that 	
\begin{equation}
\label{eq45}C_1(\overline{\varphi}_\lambda,y_0) \neq 0.
\end{equation}

On the other hand hypothesis $H_1\,(iv)$ and Proposition \ref{prop6} of Leonardi-Papageorgiou \cite{Ref9} imply that
\begin{equation}
\label{eq46}C_k(\overline{\varphi}_\lambda,0) = 0 \quad \mbox{for all $k \in \mathbb{N}_0$}.
\end{equation}

Comparing \eqref{eq45} and \eqref{eq46}, we conclude that $y_0 \neq 0$ and so $y_0 \in C_0^1(\overline{\Omega}) $ is a nodal solution of \eqref{eq0}.
	\end{proof}

This also proves  Theorem~\ref{theo11}. \qed

\section*{Acknowledgements} 
The authors wish to thank a knowledgeable referee for his/her constructive remarks and criticisms.
The second author was supported by Slovenian Research Agency
grants P1-0292, N1-0114, N1-0083, N1-0064, and J1-8131.

\end{document}